\documentclass{article}
\usepackage{fullpage}
\usepackage[usenames]{color}
\usepackage{bm,bbm,amsmath,amssymb,amsthm,graphicx}
\usepackage[caption=false,font=normalsize,labelfont=sf,textfont=sf]{subfig}
\usepackage[title]{appendix}
\usepackage{algorithm,pifont}
\usepackage{algpseudocode}
\usepackage{todonotes}
\usepackage{hyperref}
\usepackage{stmaryrd}
\usetikzlibrary{calc,patterns,decorations.pathmorphing,decorations.markings}

\allowdisplaybreaks

\numberwithin{equation}{section}

\newcommand{\MAT}{\left[ \begin{array}}  
\newcommand{\mat}{\end{array} \right]}
\newtheorem{Definition}{Definition}[section]
\newtheorem{Corollary}{Corollary}[section]
\newtheorem{Lemma}{Lemma}[section]
\newtheorem{Theorem}{Theorem}[section]

\newtheorem{Q}{Question}[section]

\newtheorem{remark}{Remark}[section]

\def \C {\mathbbm{C}}

\def \R {\mathbbm{R}}

\def \u {{\bf u}}

\def \v {{\bf v}}

\def \x {{\bf x}}

\def \y {{\bf y}}

\begin{document}

\title{On Outer Bi-Lipschitz Extensions of Linear Johnson-Lindenstrauss Embeddings of Low-Dimensional Submanifolds of $\mathbbm{R}^N$}

\author{Mark A. Iwen\thanks{Department of Mathematics, and Department of Computational Mathematics, Science and Engineering (CMSE), Michigan State University, East Lansing, MI 48824.  {\bf E-mail:} iwenmark@msu.edu} ~~and~ Mark Philip Roach\thanks{Department of Mathematics, Michigan State University, East Lansing, MI 48824.  {\bf E-mail:}  roachma3@msu.edu}}
%

\maketitle
\begin{abstract}
Let $\mathcal{M}$ be a compact $d$-dimensional submanifold of $\mathbbm{R}^N$ with reach $\tau$ and volume $V_{\mathcal M}$.  Fix $\epsilon \in (0,1)$.  In this paper we prove that a nonlinear function $f: \mathbbm{R}^N \rightarrow \mathbbm{R}^{m}$ exists with $m \leq C \left(d / \epsilon^2 \right) \log \left(\frac{\sqrt[d]{V_{\mathcal M}}}{\tau} \right)$ such that 
$$(1 - \epsilon) \| {\bf x} - {\bf y} \|_2 \leq \left\| f({\bf x}) - f({\bf y}) \right\|_2 \leq (1 + \epsilon) \| {\bf x} - {\bf y} \|_2$$
holds for all ${\bf x} \in \mathcal{M}$ and ${\bf y} \in \mathbbm{R}^N$.  In effect, $f$ not only serves as a bi-Lipschitz function from $\mathcal{M}$ into $\mathbbm{R}^{m}$ with
bi-Lipschitz constants close to one, but also approximately preserves all distances from points not in $\mathcal{M}$ to all points in $\mathcal{M}$ in its image.  Furthermore, the proof is constructive and yields an algorithm which works well in practice.  In particular, it is empirically demonstrated herein that such nonlinear functions allow for more accurate compressive nearest neighbor
classification than standard linear Johnson-Lindenstrauss embeddings do in practice.  
\end{abstract}

\section{Introduction}

The classical Kirszbraun theorem \cite{kirszbraun1934zusammenziehende} ensures that a Lipschitz continuous function $f: S \rightarrow \mathbbm{R}^m$ from a subset $S \subset \mathbbm{R}^N$ into $\mathbbm{R}^m$ can always be extended to a function $\tilde{f}: \mathbbm{R}^N \rightarrow \mathbbm{R}^m$ with the same Lipschitz constant as $f$.  More recently, similar results have been proven for bi-Lipschitz functions, $f: S \rightarrow \mathbbm{R}^m$, from $S\subset \mathbbm{R}^N$ into $\mathbbm{R}^m$ in the theoretical computer science literature.  In particular, it was shown in \cite{MMMR2018} that outer extensions of such bi-Lipschitz functions $f$, $\tilde{f}: \mathbbm{R}^N \rightarrow \mathbbm{R}^{m+1}$, exist which both $(i)$ approximately preserve $f$'s Lipschitz constants, and which $(ii)$ satisfy $\tilde{f}({\bf x}) = (f({\bf x}),0)$ for all ${\bf x} \in S$.  Narayanan and Nelson \cite{NN2019} then applied similar outer extension methods to a special class of the linear bi-Lipschitz maps guaranteed to exist for any given finite set $S \subset \mathbbm{R}^N$ by Johnson-Lindenstrauss (JL) lemma \cite{johnson1984extensions} in order prove the following remarkable result:  For each finite set $S \subset \mathbbm{R}^N$ and $\epsilon \in (0,1)$ there exists a {\bf terminal embedding of $S$}, $f: \mathbbm{R}^N \rightarrow \mathbbm{R}^{\mathcal{O} \left( \log |S| / \epsilon^2 \right)}$, with the property that
\begin{equation}
(1 - \epsilon) \| {\bf x} - {\bf y} \|_2 \leq \left\| f({\bf x}) - f({\bf y}) \right\|_2 \leq (1 + \epsilon) \| {\bf x} - {\bf y} \|_2
\label{OptimalTermEmbed}
\end{equation}
holds $\forall {\bf x} \in S$ and $\forall {\bf y} \in \mathbbm{R}^N$.  

In this paper we generalize Narayanan and Nelson's theorem for finite sets to also hold for infinite subsets $S \subset \mathbbm{R}^N$, and then give a specialized variant for the case where the infinite subset $S \subset \mathbbm{R}^N$ in question is a compact and smooth submanifold of $\mathbbm{R}^N$.  As we shall see below, generalizing this result requires us to both alter the bi-Lipschitz extension methods of \cite{MMMR2018} as well as to replace the use of embedding techniques utilizing cardinality in \cite{NN2019} with different JL-type embedding methods involving alternate measures of set complexity which remain meaningful for infinite sets (i.e., the Gaussian width of the unit secants of the set $S$ in question).  In the special case where $S$ is a submanifold of $\mathbbm{R}^N$, recent results bounding the Gaussian widths of the unit secants of such sets in terms of other fundamental geometric quantities (e.g., their reach, dimension, volume, etc.) \cite{DBLP:journals/corr/abs-2110-04193} can then be brought to bear in order to produce terminal manifold embeddings of $S$ into $\mathbbm{R}^m$ satisfying \eqref{OptimalTermEmbed} with $m$ near-optimally small.

Note that a non-trivial terminal embedding, $f$, of $S$ satisfying \eqref{OptimalTermEmbed} for all ${\bf x} \in S$ and ${\bf y} \in \mathbbm{R}^N$ must be nonlinear.  In contrast, prior work on bi-Lipschitz maps of submanifolds of $\mathbbm{R}^N$ into lower dimensional Euclidean space in the mathematical data science literature have all utilized {\it linear} maps (see, e.g., \cite{baraniuk2009random,eftekhari2015new,DBLP:journals/corr/abs-2110-04193}).  As a result, it is impossible for such previously considered linear maps to serve as terminal embeddings of submanifolds of $\mathbbm{R}^N$ into lower-dimensional Euclidean space without substantial modification.  Another way of viewing the work carried out herein is that it constructs outer bi-Lipschitz extensions of such prior linear JL embeddings of manifolds in a way that effectively preserves their near-optimal embedding dimension in the final resulting extension.  Motivating applications of terminal embeddings of submanifolds of $\mathbbm{R}^N$ related to compressive classification via manifold models \cite{davenport2007smashed} are discussed next.

\subsection{Universally Accurate Compressive Classification via Noisy Manifold Data}
\label{sec:motivating_application}

It is one of the sad facts of life that most everyone eventually comes to accept:  everything living must eventually die, you can't always win, you aren't always right, and -- worst of all to the most dedicated of data scientists -- there is always noise contaminating your datasets.  Nevertheless, there are mitigating circumstances and achievable victories implicit in every statement above -- most pertinently here, there are mountains of empirical evidence that noisy training data still permits accurate learning.  In particular, when the noise level is not too large, the mere existence of a low-dimensional data model which only approximately fits your noisy training data can still allow for successful, e.g., nearest-neighbor classification using only a highly compressed version of your original training dataset (even when you know very little about the model specifics) \cite{davenport2007smashed}.  Better quantifying these empirical observations in the context of low-dimensional manifold models is the primary motivation for our main result below.

For example, let $\mathcal{M} \subset \mathbbm{R}^N$ be a $d$-dimensional submanifold of $\mathbbm{R}^N$ (our data model), fix $\delta \in \mathbbm{R}^+$ (our effective noise level), and choose $T \subseteq {\rm tube}(\delta,\mathcal{M}) := \left \{ {\bf x} ~|~ \exists {\bf y} \in \mathcal{M} ~{\rm with}~ \| {\bf x} - {\bf y} \|_2 \leq \delta \right\}$ (our ``noisy'' and potentially high-dimensional training data).  Fix $\epsilon \in (0,1)$.  For a terminal embedding $f: \mathbbm{R}^N \rightarrow \mathbbm{R}^m$ of $\mathcal{M}$ as per \eqref{OptimalTermEmbed}, one can see that
\begin{equation}
 (1 - \epsilon) \left\| {\bf z} - {\bf t} \right\|_2 -2(1-\epsilon) \delta \leq \left\| f({\bf z}) - f( {\bf t}) \right\|_2 \leq (1 + \epsilon) \left\| {\bf z} - {\bf t} \right\|_2 + 2(1+\epsilon) \delta 
\label{equ:DisttoMan}
\end{equation}
will hold simultaneously for all ${\bf z} \in \mathbbm{R}^N$ and ${\bf t} \in T$, {\it where $f$ has an embedding dimension that only depends on the geometric properties of ${\mathcal M}$ (and not necessarily on $|T|$)}.\footnote{One can prove \eqref{equ:DisttoMan} by comparing both ${\bf z}$ and ${\bf t}$ to a point ${\bf x_t} \in \mathcal{M}$ satisfying $\| {\bf t} - {\bf x_t} \|_2 \leq \delta$ via several applications of the (reverse) triangle inequality.}  Thus, if $T$ includes a sufficiently dense external cover of $\mathcal{M}$, then $f$ will allow us to approximate the distance of all ${\bf z} \in \mathbbm{R}^N$ to $\mathcal{M}$ in the compressed embedding space via the estimator 
\begin{equation}
\tilde{d}(f({\bf z}),f(T)) := \inf_{{\bf t } \in T} \left\| f({\bf z}) - f( {\bf t}) \right\|_2 \approx d({\bf z},\mathcal{M}) := \inf_{{\bf y} \in \mathcal{M}} \| {\bf z} - {\bf y} \|_2
\label{equ:CompressedEstimator}
\end{equation} 
up to $\mathcal{O}(\delta)$-error.  As a result, if one has noisy data from two disjoint manifolds $\mathcal{M}_1,\mathcal{M}_2 \subset \mathbbm{R}^N$, one can use this compressed $\tilde{d}$ estimator to correctly classify all data ${\bf z} \in {\rm tube}(\delta,\mathcal{M}_1) \bigcup {\rm tube}(\delta,\mathcal{M}_2)$ as being in either $T_1 := {\rm tube}(\delta,\mathcal{M}_1)$ (class 1) or $T_2 := {\rm tube}(\delta,\mathcal{M}_2)$ (class 2) as long as $\displaystyle \inf_{{\bf x} \in T_1, {\bf y} \in T_2} \| {\bf x} - {\bf y} \|_2$ is sufficiently large.  In short, terminal manifold embeddings demonstrate that accurate compressive nearest-neighbor classification based on noisy manifold training data is always possible as long as the manifolds in question are sufficiently far apart (though not necessarily separable from one another by, e.g., a hyperplane, etc.).  

Note that in the discussion above we may in fact take $T = {\rm tube}(\delta,\mathcal{M})$. In that case \eqref{equ:DisttoMan} will hold simultaneously for all ${\bf z} \in \mathbbm{R}^N$ and $({\bf t}, \delta) \in \mathbbm{R}^N \times \mathbbm{R}^+$ with ${\bf t} \in {\rm tube}(\delta,\mathcal{M})$ so that $f:  \mathbbm{R}^N \rightarrow \mathbbm{R}^m$ will approximately preserve the distances of all points ${\bf z} \in \mathbbm{R}^N$ to ${\rm tube}(\delta,\mathcal{M})$ up to errors on the order of $\mathcal{O}(\epsilon) d({\bf z}, {\rm tube}(\delta,\mathcal{M})) + \mathcal{O}(\delta)$ for all $\delta \in \mathbbm{R}^+$.  This is in fact rather remarkable when one recalls that the best achievable embedding dimension, $m$, here only depends on the geometric properties of the low-dimensional manifold $\mathcal{M}$ (see Theorem~\ref{Thm:MainResult} for a detailed accounting of these dependences).  

We further note that alternate applications of Theorem~\ref{Thm:OptimalExistence} (on which Theorem~\ref{Thm:MainResult} depends) involving other data models are also possible.  As a more explicit second example, suppose that $\mathcal{M}$ is a union of $n$ $d$-dimensional affine subspaces so that its unit secants, $S_{\mathcal{M}}$ defined as per \eqref{equ:UnitSecs}, are contained in the union of at most ${n \choose 2} + n$ unit spheres $\subset \mathbb{S}^{N-1}$, each of dimension at most $2d+1$.  The Gaussian width (see Definition~\ref{def:GaussianWidth}) of $S_{\mathcal{M}}$ can then be upper-bounded by $C \sqrt{d+\log n}$ using standard techniques, where $C \in \mathbbm{R}^+$ is an absolute constant.  
An application of Theorem~\ref{Thm:OptimalExistence} now guarantees the existence of a terminal embedding $f: \mathbbm{R}^N \rightarrow \mathbbm{R}^{\mathcal{O}\left(\frac{d + \log n}{\epsilon^2}\right)}$ which will allow approximate nearest subspace queries to be answered for any input point ${\bf z} \in \mathbbm{R}^N$ using only $f({\bf z})$ in the compressed $\mathcal{O}\left(\frac{d + \log n}{\epsilon^2}\right)$-dimensional space.  Even more specifically, if we choose, e.g., $\mathcal{M}$ to consist of all at most $s$-sparse vectors in $\mathbbm{R}^N$ (i.e., so that $\mathcal{M}$ is the union of $n = {N \choose s}$ subspaces of $\mathbbm{R}^N$), we can now see that Theorem~\ref{Thm:OptimalExistence} guarantees the existence of a deterministic compressed estimator \eqref{equ:CompressedEstimator} which allows for the accurate approximation of the best $s$-term approximation error $\displaystyle \inf_{{\bf y} \in \mathbbm{R}^N ~\textrm{at most}~s~\textrm{sparse}} \| {\bf z} - {\bf y} \|_2$ for all ${\bf z} \in \mathbbm{R}^N$ using only $f({\bf z}) \in \mathbbm{R}^{\mathcal{O}(s \log (N/s))}$ as input.  Note that this is only possible due to the non-linearity of $f$ herein.  In, e.g., the setting of classical compressive sensing theory where $f$ must be linear it is known that such good performance is impossible \cite[Section 5]{cohen2009compressed}.

\subsection{The Main Result and a Brief Outline of Its Proof}

The following theorem is proven in Section~\ref{sec:MainThmProof}.  Given a low-dimensional submanifold $\mathcal{M}$ of $\mathbbm{R}^N$ it establishes the existence of a function $f: \mathbbm{R}^N \rightarrow \mathbbm{R}^m$ with $m \ll N$ that approximately preserves the Euclidean distances from all points in $\mathbbm{R}^N$ to all points in $\mathcal{M}$.  As a result, it guarantees the existence of a low-dimensional embedding which will, e.g., always allow for the correct compressed nearest-neighbor classification of images living near different well separated submanifolds of Euclidean space.

\begin{Theorem}[The Main Result] \label{Thm:MainResult}
Let $\mathcal{M}  \hookrightarrow \mathbbm{R}^N$ be a compact $d$-dimensional submanifold of $\mathbbm{R}^N$ 
with boundary $\partial \mathcal{M}$, finite reach $\tau_{\mathcal{M}}$ (see Definition~\ref{def:Reach}), and volume  $V_{\mathcal M}$.  Enumerate the connected components of $\partial \mathcal{M}$ and let $\tau_i$ be the reach of the $i^{\rm th}$ connected component of $\partial \mathcal M$ as a submanifold of $\mathbbm{R}^N$. 
Set $\tau := \min_{i} \{\tau_{\mathcal M}, \tau_i \}$, let $V_{\partial \mathcal{M}}$ be the volume of $\partial \mathcal{M}$, and denote the volume of the $d$-dimensional Euclidean ball of radius $1$ by $\omega_d$.  Next, 
\begin{enumerate}
    \item if $d=1$, define $\alpha_{\mathcal M} :=  \frac{20 V_{\mathcal M}}{\tau}  + V_{\partial {\mathcal M}}$, else
    \item if $d \geq 2$, define $\alpha_{\mathcal M} :=
\frac{V_\mathcal{M}}{ \omega_d} \left(\frac{41}{\tau} \right)^d 
+ \frac{V_{\partial \mathcal{M}}}{ \omega_{d-1}} \left(\frac{81}{\tau} \right)^{d-1}$.  
\end{enumerate}
Finally, fix $\epsilon \in (0,1)$ and define 
\begin{align}
\beta_{\mathcal{M}} &:= \left(\alpha_{\mathcal M}^2 +3^{d} \alpha_{\mathcal M} \right). 
\end{align}
Then, there exists a map $f: \R^N \rightarrow \C^{m}$ with 
$m \leq c \left(\ln \left(\beta_{\mathcal{M}} \right)+4d \right) / \epsilon^2$
that satisfies
\begin{equation}
\label{equ:optimaldiscussion}
\left| \left\| f(\x) - f(\y) \right\|^2_2 - \left\| \x - \y \right\|_2^2 \right| \leq \epsilon \left\| \x - \y \right\|_2^2
\end{equation}
for all $\x \in \mathcal{M}$ and $\y \in \R^N$.  Here $c \in \mathbbm{R}^+$ is an absolute constant independent of all other quantities.
\end{Theorem}

\begin{proof}
See Section \ref{sec:MainThmProof}.
\end{proof}

The remainder of the paper is organized as follows.  In Section~\ref{sec:Prelims} we review notation and state a result from \cite{DBLP:journals/corr/abs-2110-04193} that bounds the Gaussian width of the unit secants of a given submanifold of $\mathbbm{R}^N$ in terms of geometric quantities of the original submanifold.  Next, in Section~\ref{sec:ExtensionRes} we prove an optimal terminal embedding result for arbitrary subsets of $\mathbbm{R}^N$ in terms of the Gaussian widths of their unit secants by generalizing results from the computer science literature concerning finite sets \cite{MMMR2018,NN2019}.  See Theorem~\ref{Thm:OptimalExistence} therein.  We then combine results from Sections~\ref{sec:Prelims} and~\ref{sec:ExtensionRes} in order to prove our main theorem in Section \ref{sec:MainThmProof}.  Finally, in Section~\ref{sec:Numerics} we conclude by demonstrating that terminal embeddings allow for more accurate compressive nearest neighbor classification than standard linear embeddings in practice.

\section{Notation and Preliminaries}
\label{sec:Prelims}

Below $B^N_{\ell^2}({\bf \x},\gamma)$ will denote the open Euclidean ball around $\x$ of radius $\gamma$ in $\R^N$.  Given an arbitrary subset $S \subset \R^N$, we will further define 
$- S := \left\{ -\x ~\big|~ \x \in S \right\}$ and $S \pm S := \left\{ {\bf x} \pm {\bf y} ~\big|~ {\bf x}, {\bf y} \in S \right\}$.  Finally, for a given $T \subset \mathbbm{R}^N$ we will also let $\overline{T}$ denote its closure, and further define the normalization operator $U: \mathbbm{R}^N \setminus \{ {\bf 0} \} \rightarrow \mathbb{S}^{N-1}$ to be such that $U(\x) := \x / \| \x \|_2$.  With this notation in hand we can then define the {\bf unit secants of $T \subset \mathbbm{R}^N$} to be 
\begin{equation}
S_{T} := \overline{U \left( (T - T) \setminus \{ {\bf 0} \} \right)} = \overline{\left \{ \frac{\x - \y}{\| \x - \y \|_2}~\big|~  \x, \y \in T, ~\x \neq \y \right \}}.
\label{equ:UnitSecs}
\end{equation}
Note that $S_T$ is always a compact subset of the unit sphere $\mathbb{S}^{N-1} \subset \mathbbm{R}^N$, and that $S_{T} = - S_{T}$.

Herein we will call a matrix $A \in \mathbbm{C}^{m \times N}$ an {\it $\epsilon$-JL \textbf{map} of a set $T \subset \mathbbm{R}^N$ into $\mathbbm{C}^m$} if 
$$(1 - \epsilon) \| {\bf x} \|_2^2 \leq \| A{\bf x}\|_2^2 \leq (1 + \epsilon) \| {\bf x} \|_2^2$$
holds for all ${\bf x} \in T$.  Note that this is equivalent to $A \in \mathbbm{C}^{m \times N}$ having the property that
$$ \sup_{{\bf x} \in T \setminus \{ {\bf 0} \}} \left| \left\| A ({\bf x}/\| {\bf x} \|_2) \right\|^2_2 - 1 \right| = \sup_{{\bf x} \in U(T)} \left| \left\| A {\bf x} \right\|^2_2 - 1 \right| \leq \epsilon,$$
where $U(T) \subset \mathbbm{R}^N$ is the normalized version of $T \setminus \{ \bf 0 \} \subset \mathbbm{R}^N$ defined as above.  Furthermore, we will say that a matrix $A \in \mathbbm{C}^{m \times n}$ is an {\it $\epsilon$-JL \textbf{embedding} of a set $T \subset \mathbbm{R}^n$ into $\mathbbm{C}^m$} if $A$ is an $\epsilon$-JL map of 
$$T - T := \left\{ {\bf x} - {\bf y } ~\big|~ {\bf x},{\bf y} \in T \right\}$$ 
into $\mathbbm{C}^m$.  Here we will be working with random matrices which will embed any fixed set $T$ of bounded size (measured with respect to, e.g., Gaussian Width \cite{vershynin_high-dimensional_2018}) with high probability.   Such matrix distributions are often called {\bf oblivious} and discussed as randomized embeddings in the absence of any specific set $T$ since their embedding quality can be determined independently of any properties of a given set $T$ beyond its size.  In particular, the class of oblivious {\bf sub-Gaussian random matrices} having independent, isotropic, and sub-Gaussian rows will receive special attention below.  

\subsection{Some Common Measures of Set Size and Complexity with Associated Bounds}

We will denote the cardinality of a finite set $T$ by $|T|$.  For a (potentially infinite) set $T \subset \mathbbm{R}^N$ we define its {\bf radius} and {\bf diameter} to be
$${\rm rad}(T) := \sup_{{\bf x} \in T} \| {\bf x} \|_2$$
and
$${\rm diam}(T) := {\rm rad}(T-T) = \sup_{{\bf x}, {\bf y} \in T} \| {\bf x} - {\bf y} \|_2,$$
respectively.  Given a value $\delta \in \mathbbm{R}^+$, a {\bf $\delta$-cover of $T$} (also sometimes called a {\bf $\delta$-net of $T$}) will be a subset $S \subset T$ such that the following holds 
$$\forall {\bf x} \in T ~\exists {\bf y} \in S ~{\rm such~that~} \| {\bf x} - {\bf y} \|_2 \leq \delta.$$ 
The {\bf $\delta$-covering number of $T$}, denoted by $\mathcal{N}(T,\delta) \in \mathbbm{N}$, is then the smallest achievable cardinality of a $\delta$-cover of $T$.
Finally, the {\bf Gaussian width} of a set $T$ is defined as follows.

\begin{Definition}({\bf Gaussian Width} \cite[Definition 7.5.1]{vershynin_high-dimensional_2018}).  
\label{def:GaussianWidth}
The Gaussian width of a set $T \subset \mathbb{R}^N$ is  \begin{align*}
w(T) := \mathbb{E} \sup_{{\bf x} \in T} \, \langle {\bf g},{\bf x} \rangle
\end{align*}
where ${\bf g}$ is a random vector with $N$ independent and identically distributed (i.i.d.) mean $0$ and variance $1$ Gaussian entries.  For a list of useful properties of the Gaussian width we refer the reader to \cite[Proposition 7.5.2]{vershynin_high-dimensional_2018}.
\end{Definition}

Finally, reach is an extrinsic parameter of a subset $S$ of Euclidean space defined based on how far away points can be from $S$ while still having a unique closest point in $S$ \cite{federer_curvature_1959,thale_50_2008}.  The following formal definition of reach utilizes the Euclidean distance $d$ between a given point ${\bf x} \in \mathbbm{R}^N$ and subset $S \subset \mathbbm{R}^N$.
 \begin{Definition}
 \label{def:Reach}
(\textbf{Reach} \cite[Definition 4.1]{federer_curvature_1959}).  For a subset $S \subset \mathbb{R}^N$ of Euclidean space, the reach $\tau_S$ is
$$
    \tau_S := \sup \left\{ t \geq 0  ~\big|~ \, \forall {\bf x} \in \mathbbm{R}^n \text{ such that } d({\bf x},S) < t, \, {\bf x} \text{ has a unique closest point in } S \right\}.
$$
\end{Definition}

The following theorem is a restatement of Theorem 20 in \cite{DBLP:journals/corr/abs-2110-04193}.  It bounds the Gaussian width of a smooth submanifold of $\mathbbm{R}^N$ in terms of its dimension, reach, and volume.

\begin{Theorem}[Gaussian Width of the Unit Secants of a Submanifold of $\mathbbm{R}^N$, Potentially with Boundary] \label{GaussianWidthOfManifodWithBoundaryViaGunther}
Let $\mathcal{M}  \hookrightarrow \mathbbm{R}^N$ be a compact $d$-dimensional submanifold of $\mathbbm{R}^N$ 
with boundary $\partial \mathcal{M}$, finite reach $\tau_{\mathcal{M}}$, and volume  $V_{\mathcal M}$.  Enumerate the connected components of $\partial \mathcal{M}$ and let $\tau_i$ be the reach of the $i^{\rm th}$ connected component of $\partial \mathcal M$ as a submanifold of $\mathbbm{R}^N$. 
Set $\tau := \min_{i} \{\tau_{\mathcal M}, \tau_i \}$, let $V_{\partial \mathcal{M}}$ be the volume of $\partial \mathcal{M}$, and denote the volume of the $d$-dimensional Euclidean ball of radius $1$ by $\omega_d$.  Next, 
\begin{enumerate}
    \item if $d=1$, define $\alpha_{\mathcal M} :=  \frac{20 V_{\mathcal M}}{\tau}  + V_{\partial {\mathcal M}}$, else
    \item if $d \geq 2$, define $\alpha_{\mathcal M} :=
\frac{V_\mathcal{M}}{ \omega_d} \left(\frac{41}{\tau} \right)^d 
+ \frac{V_{\partial \mathcal{M}}}{ \omega_{d-1}} \left(\frac{81}{\tau} \right)^{d-1}$.  
\end{enumerate}
Finally, define 
\begin{align}
\beta_{\mathcal{M}} &:= \left(\alpha_{\mathcal M}^2 +3^{d} \alpha_{\mathcal M} \right). \label{equ:betadef}
\end{align}
Then, the Gaussian width of $\overline{U \left((\mathcal{M}-\mathcal{M}) \setminus \{ {\bf 0} \} \right)}$ satisfies $$w \left(  S_{\mathcal M} \right) = w \left(\overline{U \left((\mathcal{M}-\mathcal{M}) \setminus \{ {\bf 0} \} \right)} \right) \leq 8\sqrt{2}\sqrt{\ln \left(\beta_{\mathcal{M}} \right)+4d}.$$
\end{Theorem}

With this Gaussian width bound in hand we can now begin the proof of our main result.  The approach will be to combine Theorem~\ref{GaussianWidthOfManifodWithBoundaryViaGunther} above with general theorems concerning the existence of outer bi-Lipschitz extensions of $\epsilon$-JL embeddings of arbitrary subsets of $\mathbbm{R}^N$ into lower-dimensional Euclidean space.  These general existence theorems are proven in the next section.

\section{The Main Bi-Lipschitz Extension Results and Their Proofs}
\label{sec:ExtensionRes}

Our first main technical result guarantees that any given JL map $\Phi$ of a special subset of $\mathbb{S}^{N-1}$ related to $\mathcal{M}$ will not only be a bi-Lipschitz map from $\mathcal{M} \subset \R^N$ into a lower dimensional Euclidean space $\mathbbm{R}^m$, but will also have an outer bi-Lipschitz extension into $\mathbbm{R}^{m + 1}$.  It is useful as a means of extending particular (structured) JL maps $\Phi$ of special interest in the context of, e.g., saving on memory costs \cite{iwen2021lower}.

\begin{Theorem}
\label{thm:GeneralExtension}
Let $\mathcal{M} \subset \R^N$, $\epsilon \in (0, 1)$, and suppose that $\Phi \in \C^{m \times N}$ is an $\left( \frac{\epsilon^2}{2304} \right)$-JL map of $S_{\mathcal{M}} + S_{\mathcal{M}}$
into $\C^m$.  Then, there exists an outer bi-Lipschitz extension of $\Phi:  \mathcal{M} \rightarrow \C^m$, $f: \R^N \rightarrow \C^{m+1}$, with the property that
\begin{equation*}
\left| \left\| f(\x) - f(\y) \right\|^2_2 - \left\| \x - \y \right\|_2^2 \right| \leq \epsilon \left\| \x - \y \right\|_2^2
\end{equation*}
holds for all $\x \in \mathcal{M}$ and $\y \in \R^N$.
\end{Theorem}  

\begin{proof}
See Section \ref{sec:proofofGeneralExtension}.
\end{proof}

Looking at Theorem~\ref{thm:GeneralExtension} we can see that an $\left( \frac{\epsilon^2}{2304} \right)$-JL map of $S_{\mathcal{M}} + S_{\mathcal{M}}$ is required in order to achieve the outer extension $f$ of interest.  This result is sub-optimal in two respects.  First, the constant factor $1/2304$ is certainly not tight and can likely be improved substantially.  More importantly though is the fact that $\epsilon$ is squared in the required map distortion which means that the terminal embedding dimension, $m+1$, will have to scale sub-optimally in $\epsilon$ (see Remark~\ref{rem:suboptimaleps} below for details).  Unfortunately, this is impossible to rectify when extending arbitrary maps $\Phi$ (see, e.g., \cite{MMMR2018}).  For sub-gaussian $\Phi$ an improvement is in fact possible, however, which is the subject of our second main technical result just below.  Using specialized theory for sub-gaussian matrices it demonstrates the existence of terminal JL embeddings for arbitrary subsets of $\R^N$ which achieve an optimal terminal embedding dimension up to constants.

\begin{Theorem}
\label{Thm:OptimalExistence}
Let $\mathcal{M} \subset \R^N$ and $\epsilon \in (0, 1)$.  There exists a map $f: \R^N \rightarrow \C^{m}$ with 
$m \leq c \left( \frac{w(S_{\mathcal M})}{\epsilon} \right)^2$
that satisfies
\begin{equation}
\label{equ:optimaldiscussion}
\left| \left\| f(\x) - f(\y) \right\|^2_2 - \left\| \x - \y \right\|_2^2 \right| \leq \epsilon \left\| \x - \y \right\|_2^2
\end{equation}
for all $\x \in \mathcal{M}$ and $\y \in \R^N$.  Here $c \in \mathbbm{R}^+$ is an absolute constant independent of all other quantities.
\end{Theorem}

\begin{proof}
See Section \ref{sec:proofofOptimalExistence}.
\end{proof}

To see the optimality of the terminal embedding dimension $m$ provided by Theorem~\ref{Thm:OptimalExistence} we note that  functions $f$ which satisfy \eqref{equ:optimaldiscussion} for all $\x, \y \in \mathcal{M}$ must in fact generally scale quadratically in both $w(S_{\mathcal M})$ and $1/\epsilon$ (see \cite[Theorem 7]{iwen2021lowerben} and \cite{larsen2016optimality}).  We will now begin proving supporting results for both of the main technical theorems above.  The first supporting results pertain to the so-called {\bf convex hull distortion} of a given linear $\epsilon$-JL map.

\subsection{All Linear $\epsilon$-JL Maps Provide $\mathcal{O}(\sqrt{\epsilon})$-Convex Hull Distortion}

A crucial component involved in proving our main results involves the approximate norm preservation of all points in the convex hull of a given set bounded set $S \subset \R^N$.  Recall that the {\bf convex hull of $S\subset \C^N$} is
	\[
	\textrm{conv}(S) := \bigcup_{j=1}^{\infty} \left\{\sum_{\ell=1}^{j} \alpha_\ell {\bf x}_\ell ~\big\vert~ \x_1,\dots,\x_j \in S,\,\alpha_1,\dots,\alpha_j \in [0,1] ~\text{s.t.}~ \sum_{ \ell = 1 }^{j} \alpha_\ell = 1\right\}.
	\]
The next theorem states that each point in the convex hull of $S \subset \R^N$ can be expressed as a convex combination of at most $N+1$ points from $S$.  Hence, the convex hulls of subsets of $\R^N$ are actually a bit simpler than they first appear.
\begin{Theorem}[Carath\'eadory, see, e.g., \cite{BARANY1982141}] 
	\label{caratheadory}
	Given $S\in \R^N$, $\forall {\x}\in \text{conv}(S)$, $\exists {\y}_1,\dots,{\y}_{\tilde{N}}$, $\tilde{N} = \min(|S|, N+1)$, such that ${\bf x} = \sum_{\ell=1}^{\tilde{N}} \alpha_\ell {\y}_\ell $ for some $\alpha_1,\dots,\alpha_{\tilde{N}} \in [0,1],\, \sum_{\ell=1}^{\tilde{N}} \alpha_\ell = 1$. 
\end{Theorem}

Finally, we say that a matrix $\Phi \in \C^{m \times N}$ provides {\bf $\epsilon$-convex hull distortion for $S \subset \mathbbm{R}^N$} if 
\[ 
\left| \| \Phi \x \|_2 - \| \x \|_2 \right| \leq \epsilon
\]
holds for all $\x \in \textrm{conv}(S)$.  The main result of this subsection states that all linear $\epsilon$-JL maps can provide $\epsilon$-convex hull distortion for the unit secants of any given set.  In particular, we have the following theorem which generalizes arguments in \cite{MMMR2018} for finite sets to arbitrary and potentially infinite sets.

\begin{Theorem}
\label{thm:ConvHullDistorionisfree}
Let $\mathcal{M} \subset \R^N$, $\epsilon \in (0, 1)$, and suppose that $\Phi \in \C^{m \times N}$ is an $\left( \frac{\epsilon^2}{4} \right)$-JL map of $S_{\mathcal{M}} + S_{\mathcal{M}}$
into $\C^m$.  Then, $\Phi$ will also provide $\epsilon$-convex hull distortion for $S_{\mathcal{M}}$.
\end{Theorem}  

The proof of Theorem~\ref{thm:ConvHullDistorionisfree} depends on two intermediate lemmas.  The first lemma is a slight modification of Lemma 3 in \cite{iwen2021lower}.

\begin{Lemma}
	\label{polarizationforreal}
Let $S\subset \R^N$ and $\epsilon\in(0,1)$. Then, an $\epsilon$-JL map $\Phi \in \C^{m\times N}$ of the set 
\[S^{\prime} = \left\{\frac{{\x}}{\|{\x}\|_2} + \frac{{\y}}{\|{\y}\|_2}, \frac{{\x}}{\|{\x}\|_2} - \frac{{\y}}{\|{\y}\|_2} ~\big\vert~ {\x},{\y} \in S\right\}
\]
will satisfy 
\[
|\Re \left(\langle \Phi{\x} ,\Phi{\y} \rangle\right) - \langle{\x},{\y}\rangle| \leq 2\epsilon \|{\x}\|_2 \|{\y}\|_2 
\]
$\forall {\x},{\y}\in S$.
\end{Lemma}
\begin{proof}
If ${\x} = {\bf 0}$ or ${\y} = {\bf 0}$ the inequality holds trivially. Thus, suppose $\x,\y \neq 0$. Consider the normalizations $\u = \frac{\x}{\|\x\|_2}, \v = \frac{\y}{\|\y\|_2}$.  The polarization identities for complex/real inner products imply that
		\begin{align*}
		\left| \Re \left(\langle\Phi{\u},\Phi{\v}\right) \rangle - \langle{\u},{\v} \rangle \right| &= \frac{1}{4} \left| \Re \left( \sum_{\ell = 0}^{3} i^\ell \left\|\Phi{\u}+i^{\ell}\Phi{\v} \right\|_2^2 \right) - \left( \|{\u}+{\v}\|_2^2  - \|{\u}-{\v}\|_2^2 \right) \right|\\
		&= \frac{1}{4} \left| \left( \left\|\Phi{\u}+\Phi{\v} \right\|_2^2 - \left\|\Phi{\u}-\Phi{\v} \right\|_2^2 \right) - \left( \|{\u}+{\v}\|_2^2  - \|{\u}-{\v}\|_2^2 \right) \right|\\
		&\leq \frac{1}{4} \left( \left| \left\|\Phi{\u}+\Phi{\v} \right\|_2^2 - \|{\u}+{\v}\|_2^2 \right| + \left| \left\|\Phi{\u}-\Phi{\v} \right\|_2^2   - \|{\u}-{\v}\|_2^2 \right| \right)\\
		&\leq \frac{\epsilon}{4} \left( \|{\u}+{\v}\|_2^2 +  \|{\u}-{\v}\|_2^2  \right)
		\leq  \frac{\epsilon}{2}\left( \|{\u}\|_2 + \|{\v}\|_2\right)^2
		\leq 2\epsilon.
		\end{align*}
		The result now follows by multiplying the inequality through by $\| \x \|_2 \| \y \|_2$.
\end{proof}

Next, we see that and linear $\epsilon$-JL maps are capable of preserving the angles between the elements of the convex hull of any bounded subset $S \subset \R^N$.

\begin{Lemma}
	\label{JLforconv}
	Suppose $S \subset \overline{B^N_{\ell^2}({\bf 0},\gamma)}$ and $\epsilon \in (0,1)$. Let $\Phi \in \C^{m\times N}$ be an $\left(\frac{\epsilon}{2\gamma^2}\right)$-JL map of the set $S^{\prime}$ defined as in Lemma \ref{polarizationforreal} into $\C^m$.  Then 
	
	\begin{equation*}
		\left| \Re \left(\langle\Phi {\x},\Phi {\y}\rangle \right) - \langle{\x},{\y}\rangle \right| \leq \epsilon
	\end{equation*}
	holds $\forall {\x},{\y} \in \textrm{conv}(S)$.
\end{Lemma}

\begin{proof}
	Let ${\x},{\y}\in \text{conv}(S)$. By Theorem \ref{caratheadory}, $\exists  \left\{{\y}_i\right\}_{i=1}^{\tilde{N}},\, \left\{{\x}_i\right\}_{i=1}^{\tilde{N}} \subset S \,{\rm and} \, \left\{\alpha_\ell\right\}_{\ell=1}^{\tilde{N}}, \left\{\beta_{\ell}\right\}_{\ell=1}^{\tilde{N}} \subset[0,1]$ with $\sum_{\ell = 1}^{\tilde{N}} \alpha_\ell = \sum_{\ell = 1}^{\tilde{N}} \beta_\ell = 1$ such that
	\[
	{\x} = \sum_{\ell=1}^{\tilde{N}} \alpha_{\ell}{\x}_\ell,\, {\rm and} \, \,	{\y} = \sum_{\ell=1}^{\tilde{N}} \beta_{\ell}{\y}_\ell.
	\]
Hence, by Lemma~\ref{polarizationforreal} we have that
	\begin{align*}
	\left| \Re \left( \langle\Phi{\x},\Phi{\y}\rangle \right) - \langle{\x},{\y}\rangle \right| &= \left| \sum_{\ell=1}^{\tilde{N}} \sum_{j=1}^{\tilde{N}} \alpha_{\ell}\beta_{j} \left( \Re \left(\langle\Phi{\x_{\ell}},\Phi{\y_{j}}\rangle \right) - \langle{\x_{\ell}},{\y_{j}}\rangle \right) \right| \\
	&\leq 2 \sum_{\ell=1}^{\tilde{N}} \sum_{j=1}^{\tilde{N}} \alpha_{\ell}\beta_{j} \left(\frac{\epsilon}{2\gamma^2}\right) \|\x_{\ell}\|_2\|\y_{j}\|_2 \\
	&\leq \epsilon \left(\sum_{\ell=1}^{\tilde{N}}\alpha_{\ell}\right)\left( \sum_{j=1}^{\tilde{N}}\beta_{j}\right) ~=~ \epsilon.
	\end{align*}
	Here we have also used the mapping error $ \left(\frac{\epsilon}{2\gamma^2}\right) $ and the fact that all norms of vectors in this case will be less than $\gamma$.
\end{proof}

We are now prepared to prove Theorem~\ref{thm:ConvHullDistorionisfree}.

\subsubsection{Proof of Theorem~\ref{thm:ConvHullDistorionisfree}}
Applying Lemma~\ref{JLforconv} with $S = S_{\mathcal{M}} = S_{\mathcal{M}} \cup -S_{\mathcal{M}}$, we note that $S' = S_{\mathcal{M}} + S_{\mathcal{M}} = (S_{\mathcal{M}} \cup -S_{\mathcal{M}}) + (S_{\mathcal{M}} \cup -S_{\mathcal{M}})$ since $S \subset \mathbb{S}^{N-1}$.  Furthermore, $\gamma = 1$ in this case.  Hence, $\Phi \in \R^{m \times N}$ being an $\left( \frac{\epsilon^2}{4} \right)$-JL map of $S_{\mathcal{M}} + S_{\mathcal{M}}$ into $\R^m$ implies that 	
\begin{equation}
		\label{JLforconvineq}
		\left| \Re \left(\langle\Phi {\x},\Phi {\y}\rangle \right) - \langle{\x},{\y}\rangle \right| \leq \frac{\epsilon^2}{2}
	\end{equation}
	holds 
	$\forall {\x},{\y} \in \textrm{conv}(S_{\mathcal M}) \subset \overline{B^N_{\ell^2}({\bf 0},1)}$.  In particular, \eqref{JLforconvineq} with $\x = \y$ implies that 
	$$\left| \| \Phi {\x} \|_2 - \| \x \|_2 \right| \left| \| \Phi {\x} \|_2 + \| \x \|_2 \right| ~=~ \left| \| \Phi {\x} \|_2^2 - \| \x \|_2^2 \right| \leq \epsilon^2/2.$$
Noting that $\left| \| \Phi {\x} \|_2 + \| \x \|_2 \right| \geq \| \x \|_2$ we can see that the desired result holds automatically if $\| \x \|_2 \geq \epsilon/2$.  Thus, it suffices to assume that that $\| \x \|_2 < \epsilon/2$, but then we are also finished since $\left| \| \Phi {\x} \|_2 - \| \x \|_2 \right |\leq \max \{ \| \x \|_2, \| \Phi \x \|_2 \} \leq  \sqrt{\| \x \|_2^2 + \epsilon^2/2} < \frac{\sqrt{3}}{2} \epsilon$ will hold in that case.

\begin{remark}
\label{rem:suboptimaleps}
Though Theorem~\ref{thm:ConvHullDistorionisfree} holds for arbitrary linear maps, we note that it has suboptimal dependence on the distortion parameter $\epsilon$.  In particular, a linear $\left( \frac{\epsilon^2}{4} \right)$-JL map of an arbitrary set will generally embed that set into $\C^m$ with $m = \Omega(1/\epsilon^4)$ \cite{larsen2016optimality}.  However, it has been shown in \cite{NN2019} that sub-Gaussian matrices will behave better with high probability, allowing for outer bi-Lipschitz extensions of JL-embeddings of finite sets into $\R^m$ with $m = \mathcal{O}(1/\epsilon^2)$.  In the next subsection we generalize those better scaling results for sub-Gaussian random matrices to (potentially) infinite sets.
\end{remark}

\subsection{Sub-Gaussian Matrices and $\epsilon$-Convex Hull Distortion for Infinite Sets}

Motivated by results in \cite{NN2019} for finite sets which achieve optimal dependence on the distortion parameter $\epsilon$ for sub-Gaussian matrices, in this section we will do the same for infinite sets using results from \cite{vershynin_high-dimensional_2018}.  Our main tool will be the following result (see also \cite[Theorem 4]{DBLP:journals/corr/abs-2110-04193}).

\begin{Theorem}[See Theorem 9.1.1 and Exercise 9.1.8 in \cite{vershynin_high-dimensional_2018}]
\label{vershynin-matrix-deviation-theorem}
Let $\Phi$ be $m \times N$ matrix whose rows are independent, isotropic, and sub-Gaussian random vectors in $\mathbbm{R}^N$. Let $p \in (0,1)$ and $S \subset \mathbb{R}^N$. 
Then there exists a constant $c$ depending only on the distribution of the rows of $\Phi$ such that
\begin{align*}
\sup_{{\bf x} \in S} \left| \| \Phi {\bf x} \|_2 - \sqrt{m} \| {\bf x} \|_2   \right| 
\leq 
c \left[w(S) + \sqrt{\ln(2/p)} \cdot {\rm rad}(S) \right]
\end{align*}
holds with probability at least $1 - p$.
\end{Theorem}

The main result of this section is a simple consequence of Theorem~\ref{vershynin-matrix-deviation-theorem} together with standard results concerning Gaussian widths \cite[Proposition 7.5.2]{vershynin_high-dimensional_2018}.

\begin{Corollary}
\label{coro:subgaussiandistorion}
Let $\mathcal{M} \subset \R^N$, $\epsilon,p \in (0, 1)$, and $\Phi \in \R^{m \times N}$ be an $m \times N$ matrix whose rows are independent, isotropic, and sub-Gaussian random vectors in $\mathbbm{R}^N$.  Furthermore, suppose that $$m \geq \frac{c'}{\epsilon^2} \left( w \left(S_{\mathcal{M}} \right) + \sqrt{\ln(2/p)} \right)^2,$$
where $c'$ is a constant depending only on the distribution of the rows of $\Phi$. Then, with probability at least $1 - p$ the random matrix $\frac{1}{\sqrt{m}} \Phi$ will simultaneously be both an $\epsilon$-JL embedding of $\mathcal{M}$ into $\R^m$ and also provide $\epsilon$-convex hull distortion for $S_{\mathcal{M}}$.
\end{Corollary}

\begin{proof}
We apply Theorem~\ref{vershynin-matrix-deviation-theorem} to $S = \textrm{conv}\left(S_{\mathcal{M}} \right)$.  In doing so we note that $w \left(\textrm{conv}\left(S_{\mathcal{M}}\right) \right) = w \left(S_{\mathcal{M}} \right)$ \cite[Proposition 7.5.2]{vershynin_high-dimensional_2018}, and that ${\rm rad}\left(\textrm{conv}\left(S_{\mathcal{M}}\right) \right) = 1$ since $\textrm{conv}\left(S_{\mathcal{M}} \right) \subseteq \overline{B^N_{\ell^2}({\bf 0},1)}$.  The result will be that $\frac{1}{\sqrt{m}} \Phi$ provides $\epsilon$-convex hull distortion for $S_{\mathcal{M}}$ as long as $c' \geq c^2$.  Next, we note that providing $\epsilon$-convex hull distortion for $S_{\mathcal{M}}$ implies that $\frac{1}{\sqrt{m}} \Phi$ will also approximately preserve the $\ell_2$-norms of all the unit vectors in $S_{\mathcal{M}} \subset \textrm{conv}\left(S_{\mathcal{M}}\right)$.  In particular, $\frac{1}{\sqrt{m}} \Phi$ will be a $3 \epsilon$-JL map of $S_{\mathcal{M}}$ into $\R^m$, which in turn implies that $\frac{1}{\sqrt{m}} \Phi$ will also be a $3 \epsilon$-JL embedding of $\mathcal{M} - \mathcal{M}$ into $\R^m$ by linearity/rescaling.  Adjusting the constant $c'$ to account for the additional factor of $3$ now yields the stated result.
\end{proof}

We are now prepared to prove our general theorems regarding outer bi-Lipschitz extensions of JL-embeddings of potentially infinite sets.

\subsection{Outer Bi-Lipschitz Extension Results for JL-embeddings of General Sets}

Before we can prove our final results for general sets we will need two supporting lemmas.  They are adapted from the proofs of analogous results in \cite{MMMR2018,NN2019} for finite sets.

\begin{Lemma}
\label{lem:Getu'foranyu}
Let $\mathcal{M} \subset \R^N$, $\epsilon \in (0,1)$, and suppose that $\Phi \in \C^{m \times N}$ provides $\epsilon$-convex hull distortion for $S_{\mathcal{M}}$.  Then, there exists a function $g: \R^N \rightarrow \C^m$ such that 
\begin{equation}
\label{equ:gprop}
    \left| \Re \left( \langle g(\y), \Phi \x \rangle \right) - \langle \y, \x \rangle \right| \leq 2\epsilon \| \y \|_2 \| \x \|_2
\end{equation}
holds for all $\x \in \overline{\mathcal{M}} - \overline{\mathcal{M}}$ and $\y \in \R^N$.
\end{Lemma}

\begin{proof}
First, we note that \eqref{equ:gprop} holds trivially for $\y = {\bf 0}$ as long as $g({\bf 0}) = {\bf 0}$.  Thus, it suffices to consider nonzero $\y$. Second, we claim that it suffices to prove the existence of a function $g: \R^N \rightarrow \C^m$ that satisfies both of the following properties
\begin{enumerate}
    \item $\| g(\y) \|_2 \leq \| \y \|_2$, and 
    \item $\left| \Re \left( \langle g(\y), \Phi \x' \rangle \right) - \langle \y, \x' \rangle \right| \leq \epsilon \| \y \|_2$  for all $\x'$ in a finite $\left(\epsilon / 2 \max \{ 1,\| \Phi \|_{2 \rightarrow 2} \} \right)$-cover $\mathcal{C}$ of $S_{\mathcal M}$,
\end{enumerate}
for all $\y \in \R^N$.  To see why, fix $\y \neq {\bf 0}$, $\x \in S_{\mathcal M}$, and let $\x' \in \mathcal{C} \subset S_{\mathcal M}$ satisfy $\| \x - \x' \|_2 \leq \epsilon / 2 \max \{ 1,\| \Phi \|_{2 \rightarrow 2} \}$.  We can see that any function $g$ satisfying both of the properties above will have
\begin{align*}
    \left| \Re \left( \langle g(\y), \Phi \x \rangle \right) - \langle \y, \x \rangle \right| &= \left| \Re \left( \langle g(\y), \Phi \x' \rangle \right) + \Re \left( \langle g(\y), \Phi \left( \x - \x' \right) \rangle \right) - \langle \y, (\x - \x') \rangle - \langle \y, \x' \rangle \right|\\
    &\leq  \left| \Re \left( \langle g(\y), \Phi \x' \rangle \right) - \langle \y, \x' \rangle \right| + \left|\langle g(\y), \Phi \left( \x - \x' \right) \rangle \right|+ \left| \langle \y, (\x - \x') \rangle  \right|\\
    &\leq \epsilon \| \y \|_2 + \| g(\y) \|_2 \| \Phi \|_{2 \rightarrow 2} \| \x - \x' \|_2 + \| \y \|_2 \| \| \x - \x' \|_2
\end{align*}
where the second property was used in the last inequality above.  

Appealing to the first property above we can now also see that $\left| \Re \left( \langle g(\y), \Phi \x \rangle \right) - \langle \y, \x \rangle \right| \leq 2\epsilon \| \y \|_2$ will hold.  
Finally, as a consequence of the definition of $S_{\mathcal{M}}$, we therefore have that \eqref{equ:gprop} will hold for all $\x \in \mathcal{M} - \mathcal {M}$ and $\y \in \R^N$ whenever Properties $1$ and $2$ hold above.  Showing that \eqref{equ:gprop} holds all $\x \in \overline{\mathcal{M}} - \overline{\mathcal {M}}$ more generally can be proven by contradiction using a limiting argument combined with the fact that both the right and left hand sides of \eqref{equ:gprop} are continuous in $\x$ for fixed $\y$.  Hence, we have reduced the proof to constructing a function $g$ that satisfies both Properties $1$ and $2$ above.

Let 
\begin{align}
 \label{equ:defg}
    g(\y) &:= {\arg\min}_{\v \in \overline{B^{2m}_{\ell^2}({\bf 0},\| \y \|_2)}} ~{\max}_{{\boldsymbol \lambda} \in \overline{B^{|\mathcal{C}|}_{\ell^1} ({\bf 0},1)}} ~h_{\y}(\v,{\boldsymbol \lambda}), ~{\rm where}\\
    \label{equ:defh}
    h_{\y}(\v,{\boldsymbol \lambda}) &:= \sum_{\u \in \mathcal{C}} \left( \lambda_{\u} \left( \langle \y, \u \rangle - \Re \left( \langle \v, \Phi \u \rangle \right)\right) - \epsilon \left| \lambda_{\u} \right| \cdot \| \y \|_2 \right)
\end{align}
where we identify $\C^m$ with $\R^{2m}$ above.
Note that Property $1$ above is guaranteed by definition \eqref{equ:defg}.  Furthermore, we note that if $$\max_{{\boldsymbol \lambda} \in \{ \pm \mathbf{e}_j \}^{|\mathcal{C}|}_{j = 1} } h_{\y}(g(\y),{\boldsymbol \lambda}) ~=~ \max_{\u \in \mathcal{C}} \left( \left| \langle \y, \u \rangle - \Re \left( \langle g(\y), \Phi \u \rangle \right)\right| - \epsilon \| \y \|_2 \right) ~\leq~ {\max}_{{\boldsymbol \lambda} \in \overline{B^{|\mathcal{C}|}_{\ell^1} ({\bf 0},1)}} ~h_{\y}(g(\y),{\boldsymbol \lambda}) ~\leq~ 0$$
then Property $2$ above will hold as well.  Thus, it suffices to show that $\min_{\v \in \overline{B^{2m}_{\ell^2}({\bf 0},\| \y \|_2)}} ~{\max}_{{\boldsymbol \lambda} \in \overline{B^{|\mathcal{C}|}_{\ell^1} ({\bf 0},1)}} ~h_{\y}(\v,{\boldsymbol \lambda}) \leq 0$ always holds in order to finish the proof.  

Noting that $h_{\y}:  \R^{2m + |\mathcal{C}|} \mapsto \R$ defined in \eqref{equ:defh} is continuous, convex (affine) in $\v$, concave in $\boldsymbol \lambda$, and further noting that both $\overline{B^{|\mathcal{C}|}_{\ell^1} ({\bf 0},1)}$ and $\overline{B^{2m}_{\ell^2}({\bf 0},\| \y \|_2)}$ are compact and convex, we may apply Von Neumann's minimax theorem \cite{neumann1928theorie} to see that
$${\min}_{\v \in \overline{B^{2m}_{\ell^2}({\bf 0},\| \y \|_2)}} ~{\max}_{{\boldsymbol \lambda} \in \overline{B^{|\mathcal{C}|}_{\ell^1} ({\bf 0},1)}} ~h_{\y}(\v,{\boldsymbol \lambda}) = {\max}_{{\boldsymbol \lambda} \in \overline{B^{|\mathcal{C}|}_{\ell^1} ({\bf 0},1)}}~{\min}_{\v \in \overline{B^{2m}_{\ell^2}({\bf 0},\| \y \|_2)}} ~ h_{\y}(\v,{\boldsymbol \lambda})$$
holds.  Thus, we will in fact be finished if we can show that $\min_{\v \in \overline{B^{2m}_{\ell^2}({\bf 0},\| \y \|_2)}} ~ h_{\y}(\v,{\boldsymbol \lambda}) \leq 0$ holds for each ${\boldsymbol \lambda} \in \overline{B^{|\mathcal{C}|}_{\ell^1} ({\bf 0},1)}$.  By rescaling this in turn is implied by showing that $\forall \u \in \textrm{conv}(\mathcal{C} \cup -\mathcal{C})$ $\exists \v \in \overline{B^{2m}_{\ell^2}({\bf 0},\| \y \|_2)}$ such that 
\begin{equation}
\label{equ:finalthing}
    \left( \langle \y, \u \rangle - \Re \left( \langle \v, \Phi \u \rangle \right) - \epsilon \| \y \|_2 \right) \leq 0
\end{equation}
holds.  

To prove \eqref{equ:finalthing} for a fixed $\u \in \textrm{conv}(\mathcal{C} \cup -\mathcal{C} ) \subseteq \textrm{conv}(S_{\mathcal{M}} \cup -S_{\mathcal{M}}) = \textrm{conv}(S_{\mathcal{M}})$ and thereby establish the stated theorem, one may set $\v = \| \y \|_2 \frac{\Phi \u}{\| \Phi \u \|_2}.$  Doing so we see that the left side of \eqref{equ:finalthing} simplifies to $\langle \y, \u \rangle - \| \y \|_2 \| \Phi \u \|_2 - \epsilon \| \y \|_2$.  To finish, we note that indeed
\begin{align*}
    \langle \y, \u \rangle - \| \y \|_2 \| \Phi \u \|_2 - \epsilon \| \y \|_2 &\leq \| \y \|_2 \| \u \|_2 - \| \y \|_2 \| \Phi \u \|_2 - \epsilon \| \y \|_2\\
    &\leq \| \y \|_2 \left(  \| \u \|_2 - \| \Phi \u \|_2 - \epsilon \right) \leq 0
\end{align*}
will then hold since $\Phi$ provides $\epsilon$-convex hull distortion for $S_{\mathcal{M}}$.
\end{proof}

\begin{Lemma}
\label{lem:ExtensionlemCHDimpliesGetExtended}
Let $\mathcal{M} \subset \R^N$ be non-empty, $\epsilon \in (0,1)$, and suppose that $\Phi \in \C^{m \times N}$ provides $\epsilon$-convex hull distortion for $S_{\mathcal{M}}$.  Then, there exists an outer bi-Lipschitz extension of $\Phi$, $f: \R^N \rightarrow \C^{m+1}$, with the property that
\begin{equation}
\label{equ:outextension}
\left| \left\| f(\x) - f(\y) \right\|^2_2 - \left\| \x - \y \right\|_2^2 \right| \leq 24 \epsilon \left\| \x - \y \right\|_2^2
\end{equation}
holds for all $\x \in \mathcal{M}$ and $\y \in \R^N$.
\end{Lemma}

\begin{proof}
Given $\y \in \R^N$ let $\y_{\mathcal{M}} \in \overline{\mathcal{M}}$ satisfy $\left\| \y - \y_{\mathcal{M}} \right\|_2 = \inf_{\x \in \overline{\mathcal{M}}} \left\| \y - \x \right\|_2$.\footnote{One can see that it suffices to approximately compute $\y_{\mathcal M}$ in order to achieve \eqref{equ:outextension} up to a fixed precision.}  We define 
	\[
	f(\y) := \begin{cases}
	\left( \Phi \y, 0 \right) & \text{if } \y \in \overline{\mathcal{M}} \\
	\left( \Phi \y_{\mathcal{M}} + g(\y - \y_{\mathcal{M}}), \sqrt{\left\| \y - \y_{\mathcal{M}} \right\|_2^2 - \left\| g\left(\y - \y_{\mathcal{M}}\right) \right\|_2^2} \right) & \text{if } \y \notin \overline{\mathcal{M}} 
	\end{cases}
 	\]
where $g$ is defined as in Lemma~\ref{lem:Getu'foranyu}.  Fix $\x \in \mathcal{M}$. If $\y \in \overline{\mathcal{M}}$ then $\| f(\x) - f(\y) \|_2^2 = \| \Phi(\x - \y) \|_2^2$, and so we can see that $\left| \| f(\x) - f(\y) \|_2^2 - \| \x - \y \|_2^2 \right| \leq 3 \epsilon \| \x - \y \|_2^2$ will hold since $\Phi$ will be $3 \epsilon$-JL embedding of $\overline{\mathcal{M}} - \overline{\mathcal{M}}$ (recall the proof of Corollary~\ref{coro:subgaussiandistorion} and note the linearity of $\Phi$).  Thus, it suffices to consider a fixed $\y \notin \overline{\mathcal{M}}$.  In that case we have
\begin{align}
    \| f(\x) - f(\y) \|_2^2 &= \| \Phi(\x - \y_{\mathcal{M}}) - g\left(\y - \y_{\mathcal{M}}\right)\|_2^2 + \left\| \y - \y_{\mathcal{M}} \right\|_2^2 - \left\| g\left(\y - \y_{\mathcal{M}}\right) \right\|_2^2 \nonumber \\
    &= \left\| \y - \y_{\mathcal{M}} \right\|_2^2 + \| \Phi(\x - \y_{\mathcal{M}}) \|_2^2 - 2 \Re \left( \langle g\left(\y - \y_{\mathcal{M}}\right), \Phi(\x - \y_{\mathcal{M}}) \rangle \right) 
    \label{equ:finalextendeqproof1}
\end{align}
by the polarization identity and parallelogram law.  

Similarly we have that  
\begin{equation}
\label{equ:finalextendeqproof2}
    \| \x - \y \|_2^2 = \left\| \left(\x - \y_{\mathcal{M}} \right) - \left(\y - \y_{\mathcal{M}} \right) \right\|_2^2 =\| \y - \y_{\mathcal{M}} \|_2^2 + \| \x - \y_{\mathcal{M}} \|_2^2 - 2 \langle \y - \y_{\mathcal{M}}, \x - \y_{\mathcal{M}} \rangle.
\end{equation}
Subtracting \eqref{equ:finalextendeqproof2} from \eqref{equ:finalextendeqproof1} we can now see that
\begin{align}
    \left| \| f(\x) - f(\y) \|_2^2 - \| \x - \y \|_2^2 \right| \leq~ &\left| \| \Phi(\x - \y_{\mathcal{M}}) \|_2^2 - \| \x - \y_{\mathcal{M}} \|_2^2 \right| ~+ \nonumber \\
    &2\left|
    \Re \left( \langle g\left(\y - \y_{\mathcal{M}}\right), \Phi(\x - \y_{\mathcal{M}}) \rangle \right)  - \langle \y - \y_{\mathcal{M}}, \x - \y_{\mathcal{M}} \rangle \right| \nonumber \\
    \leq~ & 3\epsilon\| \x - \y_{\mathcal{M}} \|_2^2 ~+~ 4 \epsilon \| \y - \y_{\mathcal{M}} \|_2 \| \x - \y_{\mathcal{M}} \|_2 \nonumber \\
    \leq~ & 3\epsilon\| \x - \y_{\mathcal{M}} \|_2^2 ~+~ 2\epsilon \left( \| \y - \y_{\mathcal{M}} \|_2^2 + \| \x - \y_{\mathcal{M}} \|_2^2 \right)
    \label{equ:finalextendeqproof3}
\end{align}
where the second inequality again appeals to $\Phi$ being a $3 \epsilon$-JL embedding of $\overline{\mathcal{M}} - \overline{\mathcal{M}}$, and to Lemma~\ref{lem:Getu'foranyu}.  Considering \eqref{equ:finalextendeqproof3} we can see that
\begin{itemize}
\item $\| \y - \y_{\mathcal{M}} \|_2 \leq \| \y - \x \|_2$ by the definition of $\y_{\mathcal{M}}$, and so
\item $\| \x - \y_{\mathcal{M}} \|_2 \leq \| \x - \y \|_2 + \| \y - \y_{\mathcal{M}} \|_2 \leq 2 \| \x - \y \|_2$, and thus
\item $\| \y - \y_{\mathcal{M}} \|_2^2 + \| \x - \y_{\mathcal{M}} \|_2^2 \leq \left( \| \y - \y_{\mathcal{M}} \|_2 + \| \x - \y_{\mathcal{M}} \|_2 \right)^2 \leq 9 \| \x - \y \|_2^2$.
\end{itemize}
Using the last two inequalities above in \eqref{equ:finalextendeqproof3} now yields the stated result.
\end{proof}

We are now prepared to prove the two main results of this section.

\subsection{Proof of Theorem~\ref{thm:GeneralExtension}}
\label{sec:proofofGeneralExtension}

Apply Theorem~\ref{thm:ConvHullDistorionisfree} with $\epsilon \leftarrow \epsilon / 24$ in order obtain $\epsilon / 24$-convex hull distortion for $S_{\mathcal M}$ via $\Phi$.  Then, apply Lemma~\ref{lem:ExtensionlemCHDimpliesGetExtended}.

\subsection{Proof of Theorem~\ref{Thm:OptimalExistence}}
\label{sec:proofofOptimalExistence}

To begin we apply Corollary~\ref{coro:subgaussiandistorion} with, e.g., $p = 1/2$ to demonstrate that an $\left \lceil \frac{c''}{\epsilon^2} \left(w(S_{\mathcal M}) + \sqrt{\ln(4)} \right)^2 \right \rceil \times N$ matrix with i.i.d. standard normal random entries can provide $(\epsilon/24)$-convex hull distortion for $S_\mathcal{M}$, where $c''$ is an absolute constant.  Hence, such a matrix $\Phi$ exists.  An application of Lemma~\ref{lem:ExtensionlemCHDimpliesGetExtended} now finishes the proof.

\section{The Proof of Theorem~\ref{Thm:MainResult}}
\label{sec:MainThmProof}

We apply Theorem~\ref{Thm:OptimalExistence} together with Theorem~\ref{GaussianWidthOfManifodWithBoundaryViaGunther} to bound the Gaussian width of $S_{\mathcal{M}}$.

\section{A Numerical Evaluation of Terminal Embeddings}
\label{sec:Numerics}

\renewcommand{\algorithmicrequire}{\textbf{Input:}}
\renewcommand{\algorithmicensure}{\textbf{Output:}}


In this section we consider several variants of the optimization approach mentioned in Section 3.3 of \cite{NN2019} for implementing a terminal embedding $f: \mathbbm{R}^N \rightarrow \mathbbm{R}^{m+1}$ of a finite set $X \subset \mathbbm{R}^N$.  In effect, this requires us to implement a function satisfying two sets of constraints from \cite[Section 3.3]{NN2019} that are analogous to the two properties of $g:  \mathbbm{R}^N \rightarrow \mathbbm{C}^m$ listed at the beginning of the proof of Lemma~\ref{lem:Getu'foranyu}.  See Lines 1 and 2 of Algorithm~\ref{alg:FTE} for a concrete example of one type of constrained minimization problem solved herein to accomplish this task.  

\begin{algorithm}[H]
\caption{Terminal Embedding of a Finite Set} \label{alg:FTE}
\begin{algorithmic}
\Require
$\epsilon \in (0,1),~ X \subset \mathbbm{R}^N,~ \lvert X \rvert =: n,~ S \subset \mathbbm{R}^N,~ \lvert S\rvert =: n'$,~$m \in \mathbbm{N}$ with $m < N$,~a random matrix with i.i.d. standard Gaussian entries, $\Phi \in \mathbbm{R}^{m \times N}$, rescaled to perform as a JL embedding matrix $\Pi := \frac{1}{\sqrt{m}} \Phi$ 
\Ensure A terminal embedding of $X$, $f \in \mathbbm{R}^N \rightarrow \mathbbm{R}^{m+1}$, evaluated on $S$
\For {${\bf u} \in S$}
\State 1) Compute ${\bf x}_{NN} := \text{argmin}_{{\bf x} \in X} \; \| {\bf u} - {\bf x}\|_{2}$
\State 2)  Solve the following constrained minimization problem to compute a minimizer ${\bf u}' \in \mathbb{R}^m$
\vspace{-3mm}
\begin{align*}
\text{Minimize} \quad &h_{{\bf u},{\bf x}_{NN}}({\bf z}) := \| {\bf z}\|^2_2 + 2 \langle \Pi ({\bf u} - {\bf x}_{NN}),{\bf z} \rangle&\\
\text{subject to} \quad &\|{\bf z}\|_{2} \leq \|{\bf u} - {\bf x}_{NN}\|_{2}&\\
&\lvert \langle {\bf z}, \Pi ({\bf x} - {\bf x}_{NN}) \rangle  - \langle {\bf u} - {\bf x}_{NN}, {\bf x} - {\bf x}_{NN}\rangle\rvert  \leq  \epsilon \|{\bf u} - {\bf x}_{NN}\|_{2}\|{\bf x} - {\bf x}_{NN}\|_{2}~\forall {\bf x} \in X&
\end{align*}
\State 3) Compute $f: \mathbbm{R}^N \rightarrow \mathbbm{R}^{m+1}$ at ${\bf u}$ via
\vspace{-4mm}
\begin{align*}
f({\bf u}) :=
\begin{cases}
(\Pi {\bf u},0), \quad &{\bf u} \in X\cr
(\Pi {\bf x}_{NN} + {\bf u}', \sqrt{\|{\bf u} - {\bf x}_{NN} \|_{2}^2 - \|{\bf u}'\|_{2}^2}), \quad &{\bf u} \notin X\cr
\end{cases}
\end{align*}
\EndFor
\end{algorithmic}
\end{algorithm}

{\it  Crucially, we note that any choice ${\bf u}' \in \mathbb{R}^m$ of a ${\bf z}$ satisfying the two sets of constraints in Line 2 of Algorithm~\ref{alg:FTE} for a given ${\bf u} \in \mathbbm{R}^N$ is guaranteed to correspond to an evaluation of a valid terminal embedding of $X$ at ${\bf u}$ in Line 3.}  This leaves the choice of the objective function, $h_{{\bf u},{\bf x}_{NN}}$, minimized in Line 2 of Algorithm~\ref{alg:FTE} open to change without effecting its theoretical performance guarantees.  Given this setup, several heretofore unexplored practical questions about terminal embeddings immediately present themselves.  These include:
\begin{enumerate}

    \item Repeatedly solving the optimization problem in Line 2 of Algorithm~\ref{alg:FTE} to evaluate a terminal embedding of $X$ on $S$ is certainly more computationally expensive than simply evaluating a standard linear Johnson-Lindenstrauss (JL) embedding of $X$ on $S$ instead.  How do terminal embeddings empirically compare to standard linear JL embedding matrices on real-world data in the context of, e.g., compressive classification?  When, if ever, is their additional computational expense actually justified in practice?
    
    \item Though any choice of  objective function $h_{{\bf u},{\bf x}_{NN}}$ in Line 2 of Algorithm~\ref{alg:FTE} must result in a terminal embedding $f$ of $X$ based on the available theory, some choices probably lead to better empirical performance than others.  What's a good default choice?
    
    \item How much dimensionality reduction are terminal embeddings capable of in the context of, e.g., accurate compressive classification using real-world data?   
    
\end{enumerate}

In keeping with the motivating application discussed in Section~\ref{sec:motivating_application}above, we will explore some preliminary answers to these three questions in the context of compressive classification based on real-world data below.  

\subsection{A Comparison Criteria:  Compressive Nearest Neighbor Classification}

Given a labelled data set $\mathcal{D} \subset \mathbbm{R}^N$ with label set $\mathcal{L}$, we let $Label: \mathcal{D} \rightarrow \mathcal{L}$ denote the function which assigns the correct label to each element of the data set.  To address the three questions above we will use compressive nearest neighbor classification accuracy as a primary measure of an embedding strategy's quality.  See Algorithm~\ref{Alg:Class} for a detailed description of how this accuracy can be computed for a given data set $\mathcal{D}$.

\begin{algorithm}[H]
\caption{Measuring Compressive Nearest Neighbor Classification Accuracy} \label{Alg:Class}
\begin{algorithmic}
\Require $\epsilon \in (0,1)$, A labeled data set $\mathcal{D} \subset \mathbbm{R}^N$ split into two disjoint subsets: A training set $X \subset \mathcal{D}$ with $\lvert X \rvert =: n$, and a test set $S \subset \mathcal{D}$ with $\lvert S\rvert =: n'$, such that $S \cap X = \emptyset$.  A compressive dimension $m < N$.  
\Ensure Successful Nearest Neighbor Classification Percentage for Data Embedded in $\mathbbm{R}^{m+1}$ 
\State Fix $f: \mathbbm{R}^N \rightarrow \mathbbm{R}^{m+1}$, an embedding of the training data $X \subset \mathbbm{R}^N$ into $\mathbbm{R}^{m+1}$ satisfying
$$(1 - \epsilon) \| {\bf x} - {\bf y} \|_2 \leq \left\| f({\bf x}) - f({\bf y}) \right\|_2 \leq (1 + \epsilon) \| {\bf x} - {\bf y} \|_2$$
for all ${\bf x},{\bf y} \in X$.  [Note: this can either be a JL-embedding of $X$, or a stronger terminal embedding of $X$.]\\

\State \% {\it Embed the training data into $\mathbbm{R}^{m+1}$.}
\For{${\bf x} \in X$}
\State Compute $f({\bf x})$ using, e.g., Algorithm~\ref{alg:FTE}.
\EndFor\\

\State \% {\it Classify the test data using its embedded distance in $\mathbbm{R}^{m+1}$.}
\State $p = 0$ 
\For{${\bf u} \in S$}
\State Compute $f({\bf u})$ using, e.g., Algorithm~\ref{alg:FTE}
\State Compute ${\bf x} = \text{argmin}_{ {\bf y} \in X} \; \|f({\bf u}) -  f({\bf y})\|_{2}$
\If{$Label({\bf u}) = Label({\bf x})$}
\State $p = p + 1$
\EndIf
\EndFor\\

\State Output the Successful Classification Percentage = $\dfrac{p}{n'} \times 100\%$
\end{algorithmic}
\end{algorithm}

Note that Algorithm~\ref{Alg:Class} can be used to help us compare the quality of different embedding strategies.  For example, one can use  Algorithm~\ref{Alg:Class} to compare different choices of objective functions $h_{{\bf u},{\bf x}_{NN}}$ in Line 2 of Algorithm~\ref{alg:FTE} against one another by running Algorithm~\ref{Alg:Class} multiple times on the same training and test data sets while only varying the implementation of Algorithm~\ref{alg:FTE} each time.  This is exactly the type of approach we will use below.  Of course, before we can begin we must first decide on some labelled data sets $\mathcal{D}$ to use in our classification experiments.

\subsection{Our Choice of Training and Testing Data Sets}
\label{sec:DataSets}

Herein we consider two standard benchmark image data sets which allow for accurate uncompressed Nearest Neighbor (NN) classification.  The images in each data set can then be vectorized and embedded using, e.g., Algorithm~\ref{alg:FTE} in order to test the accuracies of compressed NN classification variants against both one another, as well as against standard uncompressed NN classification.  These benchmark data sets are as follows. 

\begin{figure}[H] 
\centering
\includegraphics[width=0.8\textwidth]{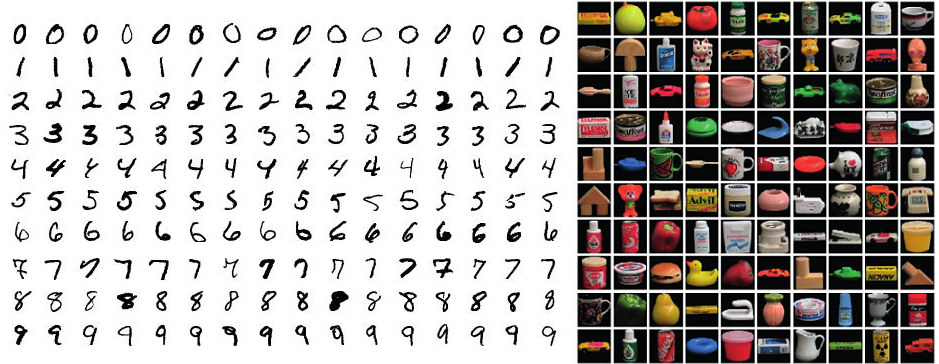}
\caption{Example images from the MNIST data set (left), and the COIL-100 data set (right).}
\label{fig:exampdata}
\end{figure}

{\bf The MNIST data set \cite{mnisthandwrittendigit-1998,deng2012mnist}} consists of 60,0000 training images of $28 \times 28$-pixel grayscale hand-written images of the digits $0$ through $9$. Thus, MNIST $10$ labels to correctly classify between, and $N = 28^2 = 784$.  For all experiments involving the MNIST dataset $n / 10$ digits of each type are selected uniformly at random to form the training set $X$, for a total of $n$ vectorized training images in $\mathbbm{R}^{784}$.  Then, $100$ digits of each type are randomly selected from those not used for training in order to form the test set $S$, leading to a total of $n' = 1000$ vectorized test images in $\mathbbm{R}^{784}$.  See the left side of Figure~\ref{fig:exampdata} for example MNIST images.\\

{\bf The COIL-100 data set \cite{nene1996columbia}} is a collection of $128 \times 128$-pixel color images of $100$ objects, each photographed $72$ times where the object has been rotated by $5$ degrees each time to get a complete rotation.  However, only the green color channel of each image is used herein for simplicity.  Thus, herein COIL-100 consists of $7,200$ total vectorized images in $\mathbbm{R}^{N}$ with $N = 128^2 = 16,384$, where each image has one of $100$ different labels (72 images per label).  For all experiments involving this COIL-100 data set, $n / 100$ training images are down sampled from each of the $100$ objects' rotational image sequences.  Thus, the training sets each contain $n / 100$ vectorized images of each object, each photographed at rotations of $\approx 36000 / n$ degrees (rounded to multiples of $5$).  The resulting training data sets therefore all consist of $n$ vectorized images in $\mathbbm{R}^{16,384}$.  After forming each training set, $10$ images of each type are then randomly selected from those not used for training in order to form the test set $S$, leading to a total of $n' = 1000$ vectorized test images in $\mathbbm{R}^{16,384}$ per experiment.  See the right side of Figure~\ref{fig:exampdata} for example COIL-100 images.

\subsection{A Comparison of Four Embedding Strategies via NN Classification}

In this section we seek to better understand $(i)$ when terminal embeddings outperform standard JL-embedding matrices in practice with respect to accurate compressive NN classification, $(ii)$ what type of objective functions $h_{{\bf u},{\bf x}_{NN}}$ in Line 2 of Algorithm~\ref{alg:FTE} perform best in practice when computing a terminal embedding, and $(iii)$ how much dimensionality reduction one can achieve with a terminal embedding without appreciably degrading standard NN classification results in practice.  To gain insight on these three questions we will compare the following four embedding strategies in the context of NN classification.  These strategies begin with the most trivial linear embeddings (i.e., the identity map) and slowly progress toward extremely non-linear terminal embeddings.
\begin{itemize}
    \item[(a)] {\bf Identity:}  We use the data in its original uncompressed form (i.e., we use the trivial embedding $f: \mathbbm{R}^N \rightarrow \mathbbm{R}^N$ defined by $f({\bf u}) = {\bf u}$ in Algorithm~\ref{Alg:Class}).  Here the embedding dimension $m+1$ is always fixed to be $N$.
    
    \item[(b)] {\bf Linear:}  We compressively embed our training data $X$ using a JL embedding.  More specifically, we generate an $m \times N$ random matrix $\Phi$ with i.i.d. standard Gaussian entries and then set
    $f: \mathbbm{R}^N \rightarrow \mathbbm{R}^{m+1}$ to be $f({\bf u}) := \left( \frac{1}{\sqrt{m}}\Phi {\bf u},~0 \right)$ in Algorithm~\ref{Alg:Class} for various choices of $m$.  It is then hoped that $f$ will embed the test data $S$ well in addition to the training data $X$.  Note that this embedding choice for $f$ is consistent with Algorithm~\ref{alg:FTE} where one lets $X = X \cup S$ when evaluating Line 3, thereby rendering the minimization problem in Line 2 irrelevant.
    
    \item [(c)]{\bf A Valid Terminal Embedding That's as Linear as Possible:}  To minimize the pointwise difference between the terminal embedding $f$ computed by Algorithm~\ref{alg:FTE} and the linear map defined above in (b), we may choose the objective function in Line 2 of Algorithm~\ref{alg:FTE} to be $h_{{\bf u},{\bf x}_{NN}}({\bf z}) := \langle \Pi ({\bf x}_{NN} - {\bf u}),{\bf z} \rangle$.  To see why solving this minimizes the pointwise difference between $f$ and the linear map in (b), let ${\bf u}'$ be such that $\langle \Pi ({\bf x}_{NN} - {\bf u}),{\bf z} \rangle$ is minimal subject to the constraints in Line 2 of Algorithm~\ref{alg:FTE} when ${\bf z} = {\bf u}'$.  Since ${\bf u}$ and ${\bf x}_{NN}$ are fixed here, we note that ${\bf z} = {\bf u}'$ will then also minimize
    \begin{align*}
    ~&~\left\| \Pi ({\bf x}_{NN} - {\bf u}) \right\|_2^2 + 2\langle \Pi ({\bf x}_{NN} - {\bf u}),{\bf z} \rangle + \| {\bf u} - {\bf x}_{NN} \|_2^2\\
    =&~\left\| \Pi ({\bf x}_{NN} - {\bf u}) \right\|_2^2 + \| {\bf z} \|_2^2+ 2\langle \Pi ({\bf x}_{NN} - {\bf u}),{\bf z} \rangle + \| {\bf u} - {\bf x}_{NN} \|_2^2 - \| {\bf z} \|_2^2\\
    =&~\left\| \Pi ({\bf x}_{NN} - {\bf u}) +{\bf z} \right\|_2^2 + \| {\bf u} - {\bf x}_{NN} \|_2^2 - \| {\bf z} \|_2^2\\
    =&~\left\| \left(\Pi {\bf x}_{NN} +{\bf z}, \sqrt{\| {\bf u} - {\bf x}_{NN} \|_2^2 - \| {\bf z} \|_2^2} \right) - (\Pi {\bf u},0) \right\|_2^2
    \end{align*}
    subject to the desired constraints.  Hence, we can see that choosing ${\bf z} = {\bf u}'$ as above is equivalent to minimizing $\| f({\bf u}) - (\Pi {\bf u},0) \|_2^2$ over all valid choices of terminal embeddings $f$ that satisfy the existing theory.
    
    \item[(d)]{\bf A Terminal Embedding Computed by Algorithm~\ref{alg:FTE} as Presented:}  This terminal embedding is computed using Algorithm~\ref{alg:FTE} exactly as it is formulated above (i.e., with the objective function in Line 2 chosen to be $h_{{\bf u},{\bf x}_{NN}}({\bf z}) := \| {\bf z}\|^2_2 + 2 \langle \Pi ({\bf u} - {\bf x}_{NN}),{\bf z} \rangle$).  Note that this choice of objective function was made to encourage non-linearity in the resulting terminal embedding $f$ computed by Algorithm~\ref{alg:FTE}.  To understand our intuition for making this choice of objective function in order to encourage non-linearity in $f$, suppose that $\| {\bf z}\|^2_2 + 2 \langle \Pi ({\bf u} - {\bf x}_{NN}),{\bf z} \rangle$ is minimal subject to the constraints in Line 2 of Algorithm~\ref{alg:FTE} when ${\bf z} = {\bf u}'$.  Since ${\bf u}$ and ${\bf x}_{NN}$ are fixed independently of ${\bf z}$ this means that ${\bf z} = {\bf u}'$ then also minimize 
    \begin{align*}
        \|{\bf z} \|^2_2 + 2 \langle \Pi ({\bf u} - {\bf x}_{NN}),{\bf z} \rangle + \|\Pi({\bf u} - {\bf x}_{NN})\|^2_2 &= \| {\bf z} + \Pi({\bf u} - {\bf x}_{NN}) \|^2_2.
    \end{align*}
    Hence, this objection function is encouraging ${\bf u}'$ to be as close to $-\Pi({\bf u} - {\bf x}_{NN}) ~=~ \Pi({\bf x}_{NN} - {\bf u})$ as possible subject to satisfying the constraints in Line 2 of Algorithm~\ref{alg:FTE}.  Recalling (c) just above, we can now see that this is exactly encouraging ${\bf u}'$ to be a value for which the objective function we seek to minimize in (c) is relatively large.
\end{itemize}

We are now prepared to empirically compare the four types of embeddings (a) -- (d) on the data sets discussed above in Section~\ref{sec:DataSets}.  To do so, we run Algorithm~\ref{Alg:Class} four times for several different choices of embedding dimension $m$ on each data set below, varying the choice of embedding $f$ between (a), (b), (c), and (d) for each value of $m$.  The successful classification percentage is then plotted as a function of $m$ for each different data set and choice of embedding.  See Figures~\ref{fig:CompareEmbeddings}(a) and~\ref{fig:CompareEmbeddings}(c) for the results.  In addition, to quantify the extent to which the embedding strategies (b) -- (d) above are increasingly nonlinear, we also measure the relative distance between where each training-set embedding $f$ maps points in the test sets versus where its associated linear training-set embedding would map them.  More specifically, for each embedding $f$ and test point ${\bf u} \in S$ we let
\begin{align*}
\text{Nonlinearity}_{f}({\bf u}) = \dfrac{\|f( {\bf u}) - (\Pi {\bf u},0)\|_{2}}{\|(\Pi {\bf u},0)\|_{2}} \times 100\%
\end{align*}
See Figures~\ref{fig:CompareEmbeddings}(b) and~\ref{fig:CompareEmbeddings}(d) for plots of $$\text {Mean}_{{\bf u} \in S} \text{ Nonlinearity}_{f}({\bf u})$$ for each of the embedding strategies (b) -- (d) on the data sets discussed in Section~\ref{sec:DataSets}.

To compute solutions to the minimization problem in Line 2 of Algorithm \ref{alg:FTE} below we used the MATLAB package CVX \cite{cvx,gb08} with the initialization ${\bf z}_0 = \Pi ( {\bf u} - {\bf x}_{NN})$ and $\epsilon = 0.1$ in the constraints. All simulations were performed using MATLAB R2021b on an Intel desktop with a 2.60GHz i7-10750H CPU and 16GB DDR4 2933MHz memory. All code used to generate the figures below is publicly available at \url{https://github.com/MarkPhilipRoach/TerminalEmbedding}.

\begin{figure}[H] 
\centering
\includegraphics[width=0.7\textwidth]{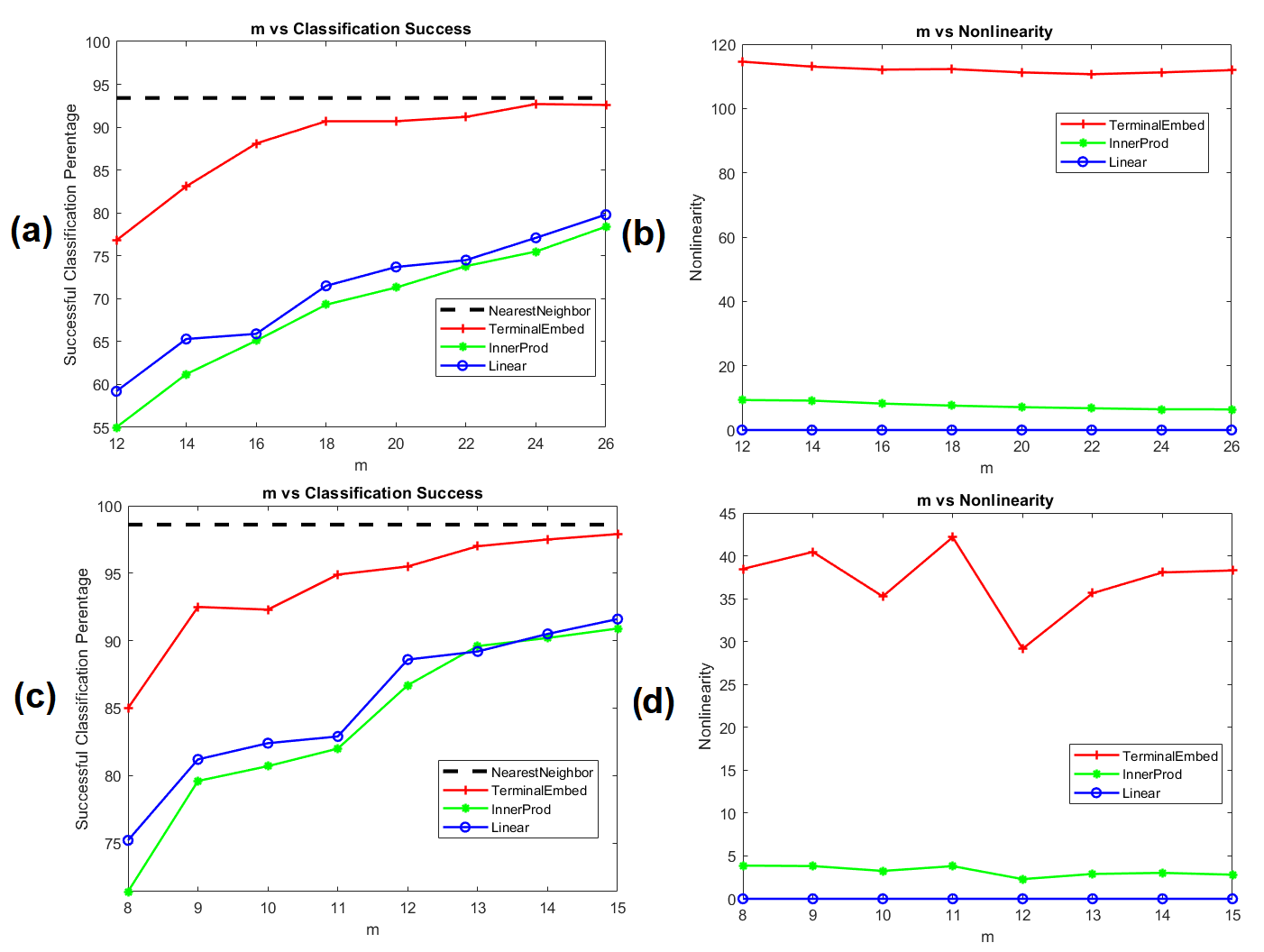}
\caption{Figures \ref{fig:CompareEmbeddings}(a) and \ref{fig:CompareEmbeddings}(b) concern the MNIST data set with training set size $n = 4000$ and test set size $n' = 1000$ in all experiments.  Similarly, Figures \ref{fig:CompareEmbeddings}(c) and \ref{fig:CompareEmbeddings}(d) concern the COIL-100 dataset with training set size $n = 3600$ and test set size $n' = 1000$ in all experiments.  In both Figures \ref{fig:CompareEmbeddings}(a) and \ref{fig:CompareEmbeddings}(c) the dashed black ``NearestNeighbor" line plots the classification accuracy when the {\bf Identity} map (a) is used in Algorithm~\ref{Alg:Class}.  Note that the ``NearestNeighbor" line is independent of $m$ because the indentity map involves no compression.  Similarly, in all of the Figures \ref{fig:CompareEmbeddings}(a) -- \ref{fig:CompareEmbeddings}(d) the red ``TerminalEmbed" curves correspond to the use of Algorithm~\ref{alg:FTE} as it's presented to compute highly non-linear terminal embeddings  (embedding strategy (d) above), the green ``InnerProd" curves correspond to the use of nearly linear terminal embeddings (embedding strategy (c) above), and the blue ``Linear" curves correspond to the use of {\bf Linear} JL embedding matrices (embedding strategy (b) above).
}
\label{fig:CompareEmbeddings}
\end{figure}

Looking at Figure~\ref{fig:CompareEmbeddings} one can see that the most non-linear embedding strategy (d) -- i.e., Algorithm~\ref{alg:FTE}
-- allows for the best compressed NN classification performance, outperforming standard linear JL embeddings for all choices of $m$.  Perhaps most interestingly, it also quickly converges to the uncompressed NN classification performance, matching it to within $1$ percent at the values of $m = 24$ for MNIST and $m = 15$ for COIL-100.  This corresponds to relative dimensionality reductions of 
$$100(1 - 24/784) \% \approx 96.9 \% $$
and
$$100(1 - 15/16384) \% \approx 99.9 \%,$$
respectively, with negligible loss of NN classification accuracy.  As a result, it does indeed appear as if nonlinear terminal embeddings have the potential to allow for improvements in dimensionality reduction in the context of classification beyond what standard linear JL embeddings can achieve.

Of course, challenges remain in the practical application of such nonlinear terminal embeddings.  Principally, their computation by, e.g., Algorithm~\ref{alg:FTE} is orders of magnitude slower than simply applying a JL embedding matrix to the data one wishes to compressively classify.  Nonetheless, if dimension reduction at all costs is one's goal, terminal embeddings appear capable of providing better results than their linear brethren.  And, recent theoretical work \cite{cherapanamjeri2022terminal} aimed at lessening their computational deficiencies looks promising.

\subsection{Additional Experiments on Effective Distortions and Run Times}

In this section we further investigate the best performing terminal embedding strategy from the previous section (i.e., Algorithm~\ref{alg:FTE}) on the MNIST and COIL-100 data sets.  In particular, we provide illustrative experiments concerning the improvement of $(i)$ compressive classification accuracy with training set size, and $(ii)$ the effective distortion of the terminal embedding with embedding dimension $m+1$.  Furthermore, we also investigate $(iii)$ the run time scaling of Algorithm~\ref{alg:FTE}.

To compute the effective distortions of a given (terminal) embedding of training data $X$, $f: \mathbbm{R}^N \rightarrow \mathbbm{R}^{m+1}$, over all available test and train data $X \cup S$ we use 
\begin{align*}
\text{MaxDist}_f = \underset{{\bf x} \in X}{\max} \; \underset{{\bf u} \in S \cup X \setminus \{ {\bf x} \}}{\max} \dfrac{\|f({\bf u}) - f({\bf x})\|_{2}}{\| {\bf u} - {\bf x}\|_{2}}, \quad \text{MinDist}_f = \underset{{\bf x} \in X}{\min} \; \underset{{\bf u} \in S \cup X \setminus \{ {\bf x} \}}{\min} \dfrac{\|f({\bf u}) - f({\bf x})\|_{2}}{\| {\bf u} - {\bf x}\|_{2}}.
\end{align*}
Note that these correspond to estimates of the upper and lower multiplicative distortions, respectively, of a given terminal embedding in \eqref{OptimalTermEmbed}. In order to better understand the effect of the minimizer ${\bf u}'$ of the minimization problem in Line 2 of Algorithm~\ref{alg:FTE} on the final embedding $f$, we will also separately consider the effective distortions of its component linear JL embedding ${\bf u} \mapsto (\Pi {\bf u}, 0)$ below.  See Figures~\ref{fig:MNISTFigure} and~\ref{fig:COILFigure} for such plots using the MNIST and COIL-100 data sets, respectively.

\begin{figure}[H] 
\centering
\includegraphics[width=1\textwidth]{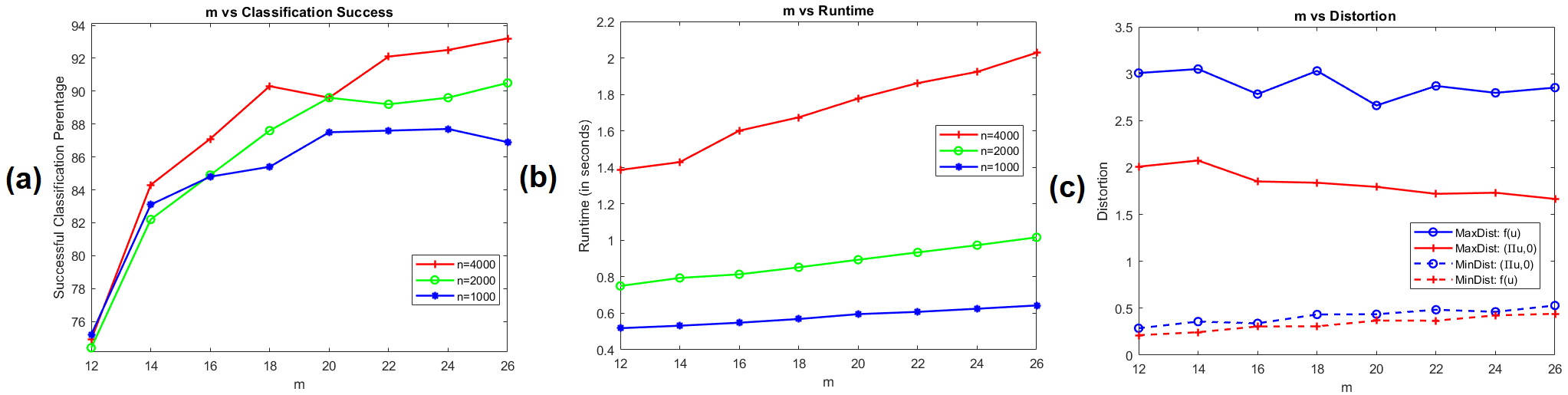}
\caption{This figure compares (a) compressive NN classification accuracies, and (b) the classification run times of Algorithm~\ref{Alg:Class} averaged over all ${\bf u} \in S$, on the MNIST data set.  Three different training data set sizes $n = \lvert X \rvert \in \{ 1000,~ 2000,~ 4000\}$ were fixed as the embedding dimension $m+1$ varied for each of the first two subfigures.  Recall that the test set size is always fixed to $n' = 1000$.  In addition, Figure (c) compares MaxDist$_f$ and MinDist$_f$ for the nonlinear $f$ computed by Algorithm \ref{alg:FTE} versus its component linear embedding ${\bf u} \mapsto (\Pi {\bf u}, 0)$ as $m$ varies for a fixed embedded training set size of $n = 4000$. 
}
\label{fig:MNISTFigure}
\end{figure}

\begin{figure}[H] 
\centering
\includegraphics[width=1\textwidth]{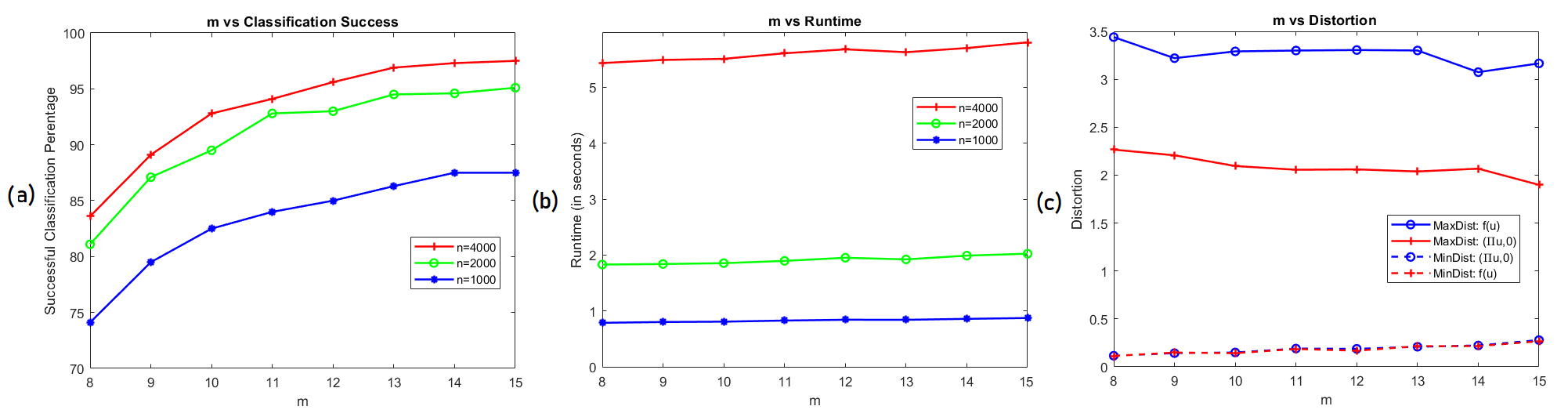}
\caption{Figures (a) and (b) here are run with identical parameters as for their corresponding subfigures in Figure~\ref{fig:MNISTFigure}, except using the COIL-100 data set.  Similarly, Figure (c) compares MaxDist$_f$ and MinDist$_f$ for the nonlinear $f$ computed by Algorithm \ref{alg:FTE} versus its component linear embedding ${\bf u} \mapsto (\Pi {\bf u}, 0)$ as $m$ varies for a fixed embedded training set size of $n = 3600$. 
}
\label{fig:COILFigure}
\end{figure}

Looking at Figures~\ref{fig:MNISTFigure} and~\ref{fig:COILFigure} one notes several consistent trends.  First, compressive classification accuracy increases with both training set size $n$ and embedding dimension $m$, as generally expected.  Second, compressive classification run times also increase with training set size $n$ (as well as more mildly with embedding dimension $m$).  This is mainly due to the increase in the number of constraints in Line 2 of Algorithm~\ref{alg:FTE} with the training set size $n$.  Finally, the distortion plots indicate that the nonlinear terminal embeddings $f$ computed by Algorithm~\ref{alg:FTE} tend to preserve the lower distortions of their component linear JL embeddings while simultaneously increasing their upper distortions.  As a result, the nonlinear terminal embeddings considered here appear to spread the initially JL embedded data out, perhaps pushing different classes away from one another in the process.  If so, it would help explained the increased compressive NN classification accuracy observed for Algorithm~\ref{alg:FTE} in Figure~\ref{fig:CompareEmbeddings}.

\section*{Acknowledgements} Mark Iwen was supported in part by NSF DMS 2106472.  Mark Philip Roach was supported in part by NSF DMS 1912706.  Mark Iwen would like to dedicate his effort on this paper to William E. Iwen, 1/27/1939 -- 12/10/2021, as well as to all those at the VA Medical Center in St. Cloud, MN who cared for him so tirelessly during the COVID-19 pandemic.\footnote{\url{https://www.greenbaypressgazette.com/obituaries/wis343656}}  Thank you.

\bibliographystyle{plain}
\bibliography{tensor}

\end{document}